\documentclass[12pt]{amsart}
\usepackage{amsmath,amsthm,amsfonts,amssymb,latexsym,color,here,graphicx}
\usepackage{enumitem,comment}
\usepackage{pgf,tikz}
\usepackage{mathrsfs}
\usetikzlibrary{arrows}

\begin{document}

\theoremstyle{plain}

\newtheorem{thm}{Theorem}[section]
\newtheorem{lem}[thm]{Lemma}
\newtheorem{pro}[thm]{Proposition}
\newtheorem{cor}[thm]{Corollary}
\newtheorem{que}[thm]{Question}
\newtheorem{rem}[thm]{Remark}
\newtheorem{defi}[thm]{Definition}
\newtheorem{prob}[thm]{Problem}

\newtheorem*{thmMain}{Main Theorem}
\newtheorem*{conjA}{Conjecture A}
\newtheorem*{thmB}{Theorem B}
\newtheorem*{thmC}{Theorem C}
\newtheorem*{thmD}{Theorem D}
\newtheorem*{thmDa}{Theorem D (a)}
\newtheorem*{thmDb}{Theorem D (b)}
\newtheorem*{thmDc}{Theorem D (c)}
\newtheorem*{thmE}{Theorem E}
\newtheorem*{probB}{Brauer's Problem 12}

\newtheorem*{thmAcl}{Main Theorem$^{*}$}
\newtheorem*{thmBcl}{Theorem B$^{*}$}

\newcommand{\Maxn}{\operatorname{Max_{\textbf{N}}}}
\newcommand{\Syl}{\operatorname{Syl}}
\newcommand{\dl}{\operatorname{dl}}
\newcommand{\Con}{\operatorname{Con}}
\newcommand{\cl}{\operatorname{cl}}
\newcommand{\Stab}{\operatorname{Stab}}
\newcommand{\Aut}{\operatorname{Aut}}
\newcommand{\Ker}{\operatorname{Ker}}
\newcommand{\fl}{\operatorname{fl}}
\newcommand{\Irr}{\operatorname{Irr}}
\newcommand{\SL}{\operatorname{SL}}
\newcommand{\FF}{\mathbb{F}}
\newcommand{\NN}{\mathbb{N}}
\newcommand{\N}{\mathbf{N}}
\newcommand{\C}{\mathbf{C}}
\newcommand{\OO}{\mathbf{O}}
\newcommand{\F}{\mathbf{F}}

\newcommand{\normal}{\lhd}
\newcommand{\bl}{{\rm bl}}
\newcommand{\cent}{{\rm \textbf{C}}}

\renewcommand{\labelenumi}{\upshape (\roman{enumi})}

\newcommand{\PSL}{\operatorname{PSL}}
\newcommand{\PSU}{\operatorname{PSU}}

\providecommand{\V}{\mathrm{V}}
\providecommand{\E}{\mathrm{E}}
\providecommand{\ir}{\mathrm{Irr_{rv}}}
\providecommand{\Irrr}{\mathrm{Irr_{rv}}}
\providecommand{\re}{\mathrm{Re}}

\def\Z{{\mathbb Z}}
\def\C{{\mathbb C}}
\def\Q{{\mathbb Q}}
\def\irr#1{{\rm Irr}(#1)}
\def\ibr#1{{\rm IBr}(#1)}
\def\irrv#1{{\rm Irr}_{\rm rv}(#1)}
\def \c#1{{\cal #1}}
\def\cent#1#2{{\bf C}_{#1}(#2)}
\def\syl#1#2{{\rm Syl}_#1(#2)}
\def\nor{\triangleleft\,}
\def\oh#1#2{{\bf O}_{#1}(#2)}
\def\Oh#1#2{{\bf O}^{#1}(#2)}
\def\zent#1{{\bf Z}(#1)}
\def\det#1{{\rm det}(#1)}
\def\ker#1{{\rm ker}(#1)}
\def\norm#1#2{{\bf N}_{#1}(#2)}
\def\alt#1{{\rm Alt}(#1)}
\def\iitem#1{\goodbreak\par\noindent{\bf #1}}
   \def \mod#1{\, {\rm mod} \, #1 \, }
\def\sbs{\subseteq}

\def\gc{{\bf GC}}
\def\ngc{{non-{\bf GC}}}
\def\ngcs{{non-{\bf GC}$^*$}}
\newcommand{\notd}{{\!\not{|}}}
\newcommand{\Out}{{\mathrm {Out}}}
\newcommand{\Mult}{{\mathrm {Mult}}}
\newcommand{\Inn}{{\mathrm {Inn}}}
\newcommand{\IBR}{{\mathrm {IBr}}}
\newcommand{\IBRL}{{\mathrm {IBr}}_{\ell}}
\newcommand{\IBRP}{{\mathrm {IBr}}_{p}}
\newcommand{\ord}{{\mathrm {ord}}}
\def\id{\mathop{\mathrm{ id}}\nolimits}
\renewcommand{\Im}{{\mathrm {Im}}}
\newcommand{\Ind}{{\mathrm {Ind}}}
\newcommand{\diag}{{\mathrm {diag}}}
\newcommand{\soc}{{\mathrm {soc}}}
\newcommand{\End}{{\mathrm {End}}}
\newcommand{\sol}{{\mathrm {sol}}}
\newcommand{\Hom}{{\mathrm {Hom}}}
\newcommand{\Mor}{{\mathrm {Mor}}}
\newcommand{\Mat}{{\mathrm {Mat}}}
\def\rank{\mathop{\mathrm{ rank}}\nolimits}
\newcommand{\Tr}{{\mathrm {Tr}}}
\newcommand{\tr}{{\mathrm {tr}}}
\newcommand{\Gal}{{\it Gal}}
\newcommand{\Spec}{{\mathrm {Spec}}}
\newcommand{\ad}{{\mathrm {ad}}}
\newcommand{\Sym}{{\mathrm {Sym}}}
\newcommand{\Char}{{\mathrm {char}}}
\newcommand{\pr}{{\mathrm {pr}}}
\newcommand{\rad}{{\mathrm {rad}}}
\newcommand{\abel}{{\mathrm {abel}}}
\newcommand{\codim}{{\mathrm {codim}}}
\newcommand{\ind}{{\mathrm {ind}}}
\newcommand{\Res}{{\mathrm {Res}}}
\newcommand{\Ann}{{\mathrm {Ann}}}
\newcommand{\Ext}{{\mathrm {Ext}}}
\newcommand{\Alt}{{\mathrm {Alt}}}
\newcommand{\AAA}{{\sf A}}
\newcommand{\SSS}{{\sf S}}
\newcommand{\CC}{{\mathbb C}}
\newcommand{\CB}{{\mathbf C}}
\newcommand{\RR}{{\mathbb R}}
\newcommand{\QQ}{{\mathbb Q}}
\newcommand{\ZZ}{{\mathbb Z}}
\newcommand{\NB}{{\mathbf N}}
\newcommand{\OB}{{\mathbf O}}
\newcommand{\ZB}{{\mathbf Z}}
\newcommand{\EE}{{\mathbb E}}
\newcommand{\PP}{{\mathbb P}}
\newcommand{\GC}{{\mathcal G}}
\newcommand{\HC}{{\mathcal H}}
\newcommand{\GA}{{\mathfrak G}}
\newcommand{\TC}{{\mathcal T}}
\newcommand{\SC}{{\mathcal S}}
\newcommand{\RC}{{\mathcal R}}
\newcommand{\GCD}{\GC^{*}}
\newcommand{\TCD}{\TC^{*}}
\newcommand{\FD}{F^{*}}
\newcommand{\GD}{G^{*}}
\newcommand{\HD}{H^{*}}
\newcommand{\GCF}{\GC^{F}}
\newcommand{\TCF}{\TC^{F}}
\newcommand{\PCF}{\PC^{F}}
\newcommand{\GCDF}{(\GC^{*})^{F^{*}}}
\newcommand{\RGTT}{R^{\GC}_{\TC}(\theta)}
\newcommand{\RGTA}{R^{\GC}_{\TC}(1)}
\newcommand{\Om}{\Omega}
\newcommand{\eps}{\epsilon}
\newcommand{\al}{\alpha}
\newcommand{\chis}{\chi_{s}}
\newcommand{\sigmad}{\sigma^{*}}
\newcommand{\PA}{\boldsymbol{\alpha}}
\newcommand{\gam}{\gamma}
\newcommand{\lam}{\lambda}
\newcommand{\la}{\langle}
\newcommand{\ra}{\rangle}
\newcommand{\hs}{\hat{s}}
\newcommand{\htt}{\hat{t}}
\newcommand{\tn}{\hspace{0.5mm}^{t}\hspace*{-0.2mm}}
\newcommand{\ta}{\hspace{0.5mm}^{2}\hspace*{-0.2mm}}
\newcommand{\tb}{\hspace{0.5mm}^{3}\hspace*{-0.2mm}}
\def\skipa{\vspace{-1.5mm} & \vspace{-1.5mm} & \vspace{-1.5mm}\\}
\newcommand{\tw}[1]{{}^#1\!}
\renewcommand{\mod}{\bmod \,}

\marginparsep-0.5cm

\renewcommand{\thefootnote}{\fnsymbol{footnote}}
\footnotesep6.5pt

\title{$p$-Blocks relative to a character of a normal subgroup}

\author{Noelia Rizo}
\address{Departament de Matem\`atiques, Universitat de Val\`encia, 46100 Burjassot, Val\`encia, Spain}
\email{noelia.rizo@uv.es}
\thanks{This research is supported by Proyecto
MTM2016-76196-P and a Fellowship
FPU of Ministerio de Educaci\'on, Cultura y Deporte}

\keywords{Height Zero Conjecture, $k(B)$-Conjecture}

\subjclass[2010]{Primary 20D; Secondary 20C15}

\begin{abstract} Let 
$G$ be a finite group,
let $N\nor G$, and let $\theta \in \irr N$
be a $G$-invariant character.
 We fix a prime $p$, and we introduce a canonical partition of ${\rm Irr}(G|\theta)$ relative to $p$. We call each member $B_\theta$
of this partition a $\theta$-block, and to each
$\theta$-block $B_\theta$ we naturally associate a conjugacy class of $p$-subgroups
of $G/N$, which we call the  $\theta$-defect groups of $B_\theta$. 
If $N$ is trivial, then the $\theta$-blocks are the Brauer $p$-blocks.
Using $\theta$-blocks, we can unify the Gluck-Wolf-Navarro-Tiep theorem
and Brauer's Height Zero conjecture in a single statement,
which, after work of B. Sambale,  turns out to be
equivalent to the the Height Zero conjecture. We also prove that the $k(B)$-conjecture is true if and only if every
$\theta$-block $B_\theta$ has size less than or equal the size of any of its $\theta$-defect groups,
hence bringing normal subgroups to the $k(B)$-conjecture.

 \end{abstract}

\maketitle
\bigskip
\section{Introduction}
Many of the global-local conjectures in the representation theory of finite groups
have been proved or reduced to finite simple groups because, somewhat miraculously, 
they admit a far more general {\sl projective} version. 
This projective version always involves an irreducible complex
character $\theta$ over a normal subgroup $N$   and  a statement which only takes into account
the irreducible characters of the group lying over $\theta$. When $N$ is the trivial
group, one recovers the original conjecture.

\medskip

Our main concern in this paper is to refine
Brauer classical $p$-blocks, where $p$ is a prime, with respect
to a fixed character of a normal subgroup.
If $N$ is a normal subgroup of a finite group  $G$, $\theta \in \irr N$
is $G$-invariant, and we wish to study
the block theory of $G$ over $\theta$, sometimes  Brauer's $p$-blocks 
do not   fully perceive some aspects of the character theory of $G$
over $\theta$. Consider, for instance,
the Gluck-Wolf-Navarro-Tiep theorem (\cite{GW}, \cite{NT}), which is a crucial step in the reduction
of Brauer's Height Zero Conjecture to a problem
on simple groups (\cite{NS2}), and the fact that only one direction
of the theorem holds: if $p$ does not divide $\chi(1)/\theta(1)$ for all
$\chi \in \irr G$ lying over $\theta$, then $G/N$ has abelian Sylow $p$-subgroups.
The converse is not true.

\medskip

In this paper we fix a prime $p$, and we introduce a canonical partition of the
irreducible characters $\irr{G|\theta}$ of  $G$ that lie over $\theta$ (relative to $p$).
We call each member $B_\theta$
of this partition a {\bf $\theta$-block} of $G$, and to each
$\theta$-block $B_\theta$ we naturally associate a conjugacy class of $p$-subgroups $D_\theta/N$
of $G/N$, which we call the  {\bf $\theta$-defect groups} of $B_\theta$.
Both the $\theta$-blocks and their $\theta$-defect groups are defined in terms of some
convenient central extensions
of $G$ and some projective representations of $G$
associated with $\theta$, and a non-trivial part of this work
is to show that they are independent of any choice that has been made in order to define them.
As we shall show, each $\theta$-block $B_\theta$ is contained in a unique
Brauer $p$-block $B$ of $G$, and the $\theta$-defect group $D_\theta/N$ is contained in $DN/N$,
for some defect group  $D$  of $B$.

\medskip
Using $\theta$-blocks,  we can  propose, for instance, the following statement
that  simultaneously generalizes the
Height Zero Conjecture and the Gluck-Wolf-Navarro-Tiep theorem.
Recall that $n_p$ is the largest power of $p$ that divides the integer $n$.

\medskip

\begin{conjA}
Let $G$ be a finite group, let $N \nor G$ and let $\theta \in \irr N$ be
$G$-invariant. Suppose that $B_\theta \sbs \irr{G|\theta}$ is
a $\theta$-block with $\theta$-defect group $D_\theta/N$. Assume that $\theta$
extends to $D_\theta$.
Then $(\chi(1)/\theta(1))_p=|G:D_\theta|_p$ for all $\chi \in B_\theta$
if and only if $D_\theta/N$ is abelian.
\end{conjA}


In the important case where $N$ is central,  
we shall prove that a $\theta$-block is simply
$\irr{B|\theta}$, where $B$ is a Brauer $p$-block of $G$ 
and $\irr{B|\theta}$ is the set of the complex irreducible
characters in $B$ that lie over $\theta$, and that a $\theta$-defect group
is $DN/N$, where $D$ is a defect group of $B$.
Using this,  Conjecture A  is then equivalent to 
a projective version of the Height Zero conjecture
which seems not
to have been noticed before (\cite{mn}). Using the theory of fusion systems,
B. Sambale has shown  that in fact this projective version
of the Height Zero Conjecture is equivalent to Brauer's original one (\cite{sambale}). 
Using his work, we can prove the following.

\begin{thmB} 
Conjecture A and Brauer's Height Zero conjecture are equivalent.
\end{thmB}

\medskip

Our second motivation to introduce $\theta$-blocks is to have a better
understanding of the celebrated Brauer's $k(B)$-conjecture. As is well-known,
this deep conjecture,   that asserts that the number of ordinary
characters in a block is less than or equal
the size of its defect groups, remains not only unsolved but also
unreduced to simple groups.   While the relative version of the McKay conjecture 
 was easy to formulate, it does not seem
easy how to do the same for the $k(B)$-conjecture.
As we can see, using $\theta$-blocks, this can be done. We believe that statements
relating normal subgroups and the $k(B)$-conjecture might help
to discern on this problem.

\begin{thmC}
The $k(B)$-conjecture is true for every finite group
if and only if whenever $N \nor G$ and $\theta \in \irr N$ is
$G$-invariant, then each
$\theta$-block $B_\theta$ of $G$ has size less than or equal the size of any of its $\theta$-defect groups.
\end{thmC}

As we said, our definition of $\theta$-blocks is related to projective
representations, and therefore with blocks of twisted group
algebras. Of course, these have been studied before by many
(including  S. B. Conlon \cite{conlon}, W. Reynolds \cite{rey}, J. Humphreys \cite{hum}, E. C. Dade \cite{dade}, or A. Laradji \cite{laradji} in the $p$-solvable case).  However,
our character theoretical approach
is new and is specifically tailored to be used in the recent developments of 
the global-local counting conjectures. 

\medskip
With our approach, we are also able to
 canonically define $\theta$-Brauer characters of $G$,
which will be class functions defined on the elements
of $g \in G$ such that $gN$ is $p$-regular.
Using these, we will prove elsewhere that there is a direct  relationship
between our $\theta$-blocks and the K\"ulshammer-Robinson $N$-projective
characters defined 
in \cite{kr}. (See also \cite{zeng1}.)
In order to do so, a certain new understanding of
Brauer $p$-blocks is required. Specifically, as in the {\sl Projective Height Zero Conjecture},
we shall need to prove certain {\sl projective} results like the following one, which we believe have independent interest.

\medskip

\begin{thmD}
Suppose that $Z$ is a central subgroup of $G$, and let $\theta \in \irr Z$. 
Let $B$ be a Brauer $p$-block of $G$. Then the decomposition matrix
$D_{\theta}=(d_{\chi \varphi})$, where $\chi \in \irr{B|\theta}$ and $\varphi \in {\rm IBr}(B)$
is not of the form
$$\left(\begin{array}{cc}
*&0\\
0&*\\
\end{array}\right).$$
\end{thmD}
Easy examples show that in Theorem D, we cannot
replace $Z$ central by a $G$-invariant
character of an abelian $Z\nor G$.

\medskip
This paper is structured as follows.
In Section \ref{secreviewpblocks}, we review some facts on ordinary blocks. In Section \ref{secreviewprojrepr}, we give the necessary background on projective representations and character triples.  In Section \ref{secthetablocks}   we give the definitions of $\theta$-blocks and $\theta$-defect groups. In Section \ref{secproperties} we present some properties of the $\theta$-blocks. In particular, we prove a $\theta$-version of a classical theorem on blocks: if $\chi \in \irr{B_\theta}$, then $\chi(g)=0$ if $g_p$ is not $G$-conjugate 
to any element of $D_\theta$. In Section \ref{secthmD}, we prove Theorem D. In Section \ref{secthmC}, we prove Theorems B and C.

\medskip
Further applications of $\theta$-blocks will appear elsewhere.

\medskip
\noindent
{\bf Acknowledgement.}~~These results are part of the author's PhD thesis at the Universitat de Val\`encia, under the supervision of
G. Navarro.
Also, some of these theorems were obtained under the direction of B. Sp\"ath during a stay of the author
at the Universit\"at of Wuppertal.  We would like to  thank B. Sp\"ath for her invaluable help and
the Department of Mathematics for the warm hospitality.
We would also like to show our gratitude to G. Malle for comments that greatly improved this paper.
The final version of this paper was completed in the Mathematical Sciences Research Institute in Berkeley,
and we thank the Institute for the hospitality.

\section{Preliminaries on $p$-blocks}\label{secreviewpblocks}
\medskip
We follow the notation of \cite{Is} for characters, and the notation of 
\cite{nbook} for blocks.
We denote by $\bf R$ the ring of algebraic integers in $\C$, and we choose
a maximal ideal $M$ of $\bf R$ containing $p\bf R$. 
Let $F={\bf R}/M$, an algebraically closed field of characteristic $p$,
and let $^*:{\bf R} \rightarrow F$ be the canonical ring homomorphism.

In this paper, sometimes we identify a $p$-block $B$ of a finite group $G$ with the set 
$\irr B$ of its irreducible complex
characters. Recall that two irreducible characters $\chi, \psi \in \irr G$
of a finite group $G$ are in the same block $B$ of $G$ if and only if 
$\omega_\chi(\hat K)^*=\omega_\psi(\hat K)^*$ for every conjugacy class
$K$ of $G$,
where 
$$\omega_\chi(\hat K)= \frac{|K|\chi(x_K)}{\chi(1)},$$
and $x_K$
is any fixed element of $K$. We denote by $\lambda_B$ the corresponding 
algebra homomorphism $\zent{FG} \rightarrow F$. Thus $\lambda_B(\hat K)
=\omega_\chi(\hat K)^*$ if $\chi \in B$.
This homomorphism is also denoted by  $\lambda_\chi$.
Also $\delta(B)$ is the set of defect groups of the block $B$.

We write $G^{p'}$ for the set of $p$-regular elements in $G$,
and $\ibr G$ is the set of irreducible Brauer
characters of $G$ (calculated with respect to
our fixed maximal ideal $M$). If $B$ is a $p$-block, we write ${\rm IBr}(B)$ to denote the set of irreducible Brauer characters lying in $B$.

\medskip
We shall need some basic facts on Brauer $p$-blocks which
we prove in this section.

\begin{lem} Let $B$ be a $p$-block of $G$ and let $\mu$ be a linear character 
of $G$. Then $$\irr{\mu B}=\{\mu \chi\>|\> \chi\in \irr B\}$$ is a $p$-block. Also, $B$ and 
$\mu B$ have the same defect groups.
\label{lemamub1}
\end{lem}

\begin{proof}  Since $\mu(x)$ is a root of unity
for $x \in G$, then $\mu(x)^*\mu(x^{-1})^*=1$.
Hence, it is clear that if $\chi, \psi \in \irr G$, then
 $\lambda_\chi=\lambda_\psi$ if and only if 
$\lambda_{\mu\chi}=\lambda_{\mu\psi}$,
and the first part follows.
Now, let $K$ be a defect class of $B$ (see page 82 of \cite{nbook}). Then 
$\lambda_{B}(\hat{K})\neq 0$ and $a_{B}(\hat{K})\neq 0$. Notice that
$\lambda_{\mu B}(\hat K)= \mu(x_K)^* \lambda_B(\hat K)$,
and $a_{\mu B}(\hat K)=\mu(x_K^{-1})^*a_B(\hat K)$, where
$x_K \in K$. 
Since $\mu(x_K)^* \ne 0$,
the result follows from Corollary 4.5 of \cite{nbook}.\end{proof}

If $N$ is normal in $G$, following \cite{Is} and \cite{nbook} we view
the characters of $G/N$ as characters of $G$ containing $N$
in their kernel. By the remarks preceding Theorem 7.6 of \cite{nbook},
recall that every block of $G/N$ is contained in a block of
$G$.  If $\chi \in {\rm Irr}(G)$, then we denote by ${\rm bl}(\chi)$ the block of 
$G$ containing $\chi$.

\begin{lem} Let $Z\leq G$, with $Z=Z_p\times K$, where $K$ is a  normal 
$p'$-subgroup of $G$ and $Z_p$ is a central $p$-subgroup of $G$.  Let $\alpha\in{\rm 
Irr}(G)$ with $Z\subseteq {\rm ker}(\alpha)$ and write $\overline{\alpha}$ for 
the character $\alpha$ viewed as character of $G/Z$. 

\begin{enumerate}[label=(\alph*)]

\item
Let $\beta\in{\rm Irr}(G)$ with $Z\subseteq {\rm ker}(\beta)$. Then ${\rm bl}(
\alpha)=\bl(\beta)$ if and only if ${\rm bl}(\overline{\alpha})={\rm bl}(
\overline{\beta})$.
\item
We have that $$\delta({\rm bl}(\overline{\alpha}))=\{PZ/Z\>|\> P\in \delta(
{\rm bl}(\alpha))\}=\{P/Z_p\>|\> P\in \delta(
{\rm bl}(\alpha))\}.$$
\end{enumerate}
\label{lemaliftblock}
\end{lem}
\begin{proof}
(a)~~It is clear that if $\bl(\overline{\alpha})=\bl(\overline{\beta})$, then 
$\bl(\alpha)=\bl(\beta)$. We need to prove the converse. We proceed by induction 
on $|G|$. By Theorem 9.9(c) of \cite{nbook} we may assume that $Z$ is a central 
$p$-group. The result follows by Theorem 7.6 of \cite{nbook}. 

For (b), we proceed by induction on $|G|$. Let $\hat{\alpha}$ be the character 
$\alpha$ viewed as a character of $G/K$. By Theorem 9.9(c) of \cite{nbook}, we 
have that $$\delta(\bl(\hat{\alpha}))=\{PK/K\>|\> P\in \delta({\rm bl}(\alpha))
\}.$$ If $K>1$, since $G/Z\cong \frac{G/K}{Z/K}$, using induction we are done. 
Hence,
we may assume that $K=1$. In this case, $Z$ is a central $p$-group. The result
now follows by Theorem 9.10 of \cite{nbook}.

\end{proof}

Suppose $\alpha: \hat G \rightarrow G$ is a surjective group homomorphism with 
kernel $Z$.
If $\psi \in \irr G$,
whenever is convenient, we  denote   by $\psi^\alpha$ the unique
irreducible character of $\hat G$ such that $\psi^\alpha(x)=\psi(\alpha(x))$ 
for $x \in \hat G$.
Notice that $Z \sbs \ker{\psi^\alpha}$. 

\medskip

\begin{cor} Suppose that $\alpha: \hat{G}\rightarrow G$ is an onto homomorphism 
with ${\rm ker}(\alpha)=Z\subseteq{\rm \textbf{Z}}(\hat{G})$.
\begin{enumerate}[label=(\alph*)]
\item
 If $\chi_i\in {\rm 
Irr}(G)$, 
then $\chi_1,\chi_2$ lie in the same block of $G$ if and only if 
$\chi_1^\alpha$ and $\chi_2^\alpha$ lie in the same block of $\hat{G}$.
\item
Suppose that $L \le G$ and let $\gamma \in \irr{L}$. If $\hat L=\alpha^{-1}(L)$,
let $\hat\gamma=\gamma^{\alpha_{\hat L}} \in \irr{\hat L}$. Then
$[(\chi^\alpha)_{\hat L}, \hat\gamma]=[\chi_L, \gamma]$
for $\chi \in \irr G$.

\item
Suppose that $\chi \in \irr G$, let $B$ be the block of $\chi$, and let $\hat B$ be
the block of $\chi^\alpha$. If $\hat D$ is a defect group of $\hat B$, then
$\alpha(\hat D)$ is a defect group of $B$.
\end{enumerate}
\label{sameblock}
\end{cor}

\begin{proof}
Part (a) is a direct consequence of Lemma \ref{lemaliftblock}(a).
Part (b) is straightforward.
To prove part (c), define $\bar \alpha: \hat G/Z \rightarrow G$ to be the associated isomorphism.
Since $Z \subseteq \ker{\chi^\alpha}$, by Lemma \ref{lemaliftblock}(b),
we have that $\hat D Z/Z$ is a defect group of the block of $\chi^\alpha$
viewed as a character of $\hat G/Z$. Since $\alpha(\hat D)=\bar\alpha(\hat DZ/Z)$,
the result follows.
\end{proof}

We shall use Gallagher's Corollary 6.18 of \cite{Is}  in several parts of this 
paper.
Recall that if $N \nor G$, $\chi \in \irr G$ and $\chi_N=\theta$, then
$\beta \mapsto \beta\chi$ defines a bijection $\irr{G/N} \rightarrow 
\irr{G|\theta}$,
where $\irr{G|\theta}$ denotes the set of irreducible constituents of
the induced character $\theta^G$.

\begin{lem} Let $N\normal G$ and let $\chi\in{\rm Irr}(G)$. Suppose that 
$\chi_N=\theta\in{\rm Irr}(N)$. Let $\chi_1,\chi_2\in{\rm Irr}(G|\theta)$ and write 
$\chi_i=\beta_i \chi$, for $i=1,2$, where $\beta_i$ in ${\rm Irr}(G/N)$. Suppose 
that $\beta_1$ and $\beta_2$ lie in the same $p$-block of $G/N$. Then $\chi_1$ 
and $\chi_2$ lie in the same $p$-block of $G$.
Also,  if $\beta\in{\rm Irr}(G/N)$ and $P/N$ is a defect group of ${\rm bl}(\beta)$,
then  $P\subseteq DN$, for some  defect group $D$ of ${\rm bl}(\beta\chi)$.
\label{lema}
\end{lem}

\begin{proof}
Let $K$ be a conjugacy class of $G$ and let $x\in K$. Write $H/N={\rm 
\textbf{C}}_{G/N}(xN)$, let $L$ be the conjugacy class of $H$ containing $x$, and let 
$S$ be the conjugacy class of $G/N$ containing $xN$. Then, by Lemma 2.2 of 
\cite{NS2}, we have that 

$$\lambda_{\chi_1}(\hat{K})=\lambda_{\chi\beta_1}(\hat{K})=\lambda_{\chi_H}(
\hat{L})\lambda_{\beta_1}(\hat{S}).$$ Since $\beta_1$ and $\beta_2$ lie in the 
same $p$-block of $G/N$, we have that 

$$\lambda_{\beta_1}(\hat{S})=\lambda_{\beta_2}(\hat{S}),$$ and hence, again by 
Lemma 2.2 of \cite{NS2}
we have that
$$\lambda_{\chi_1}(\hat{K})=\lambda_{\chi_H}(\hat{L})\lambda_{\beta_1}(\hat{S})
=\lambda_{\chi_H}(\hat{L})\lambda_{\beta_2}(\hat{S})=\lambda_{\chi\beta_2}(
\hat{K})=\lambda_{\chi_2}(\hat{K}).$$ Hence $\chi_1$ and $\chi_2$ lie in the same 
$p$-block of $G$.  The second part is
Proposition 2.5(b) of \cite{NS2}.
\end{proof}

\section{Reviewing projective representations}\label{secreviewprojrepr}

To define the $\theta$-blocks we need some background on projective representations.
We follow Chapter 11 of \cite{Is} and Chapter 5 of \cite{nbook2}.
Recall that a complex {\bf projective representation} of a finite group $G$ is a map 
$$\mathcal{P}:G\rightarrow {\rm GL}_n(\mathbb{C})$$ such that for every $x,y\in G$ there is some $\alpha(x,y)\in \mathbb{C}^\times$ satisfying $$\mathcal{P}(x)\mathcal{P}(y)=\alpha(x,y)\mathcal{P}(xy).$$ The  function $\alpha:G\times G\rightarrow \mathbb{C}^\times$ is called the {\bf factor set} of $\mathcal{P}$.
\medskip

If $G$ is a finite group, $N \nor G$, and $\theta \in \irr N$
is $G$-invariant, then we
say that  $(G,N,\theta)$ is a {\bf character triple}. The theory of character triples
and their isomorphisms
was developed by I. M. Isaacs, and we refer to Chapter 11
of \cite{Is} for their properties.

If $(G,N,\theta)$
is a character triple, we say that a projective representation
of $G$ is {\bf associated} with $\theta$ if 

\begin{enumerate}[label=(\alph*)]
\item $\mathcal{P}_N$ is an ordinary representation of $N$ affording $\theta$, and 
\item $\mathcal{P}(ng)=\mathcal{P}(n)\mathcal{P}(g)$ and $\mathcal{P}(gn)=\mathcal{P}(g)\mathcal{P}(n)$ for $g\in G$ and $n\in N$. 

\end{enumerate}
\medskip

We will need the following later.

\begin{lem} Suppose that $(G,N,\theta)$ is a character triple, and let $\mathcal{P}$ be a projective representation of $G$ associated with $\theta$ with factor set $\alpha$. Then 
\begin{enumerate}[label=(\alph*)]
\item $\alpha(1,1)=\alpha(g,n)=\alpha(n,g)=1$ for $n\in N$, $g\in G$.
\item $\alpha(xn,ym)=\alpha(x,y)$ for $x,y\in G$, $n,m\in N$.
\end{enumerate}
\label{lemafactorsets}
\end{lem}

\begin{proof}
This is Lemma 11.5 and Theorem 11.7 of \cite{Is}. See also Lemma 5.3 of \cite{nbook2}.
\end{proof}

 An important fact about projective representations is that given a character triple $(G,N,\theta)$, there always exists a projective representation associated with $\theta$ such that its factor set has roots of unity values. 

\begin{thm} Let $(G,N,\theta)$ be a character triple. Then there exists a projective representation associated with $\theta$ with factor set $\alpha$ such that 

$$\alpha(x,y)^{|G|\theta(1)}=1$$ for all $x,y\in G$.
\end{thm}

\begin{proof} See for instance Theorem 8.2 of \cite{Is2} or Theorem 5.5 of \cite{nbook2}.
\end{proof}

Using such a projective representation $\mathcal P$, it is possible to associate to each
character triple $(G,N, \theta)$ a new finite group $\hat G$, a finite central
extension of $G$ which only depends on $\mathcal P$.
This finite group $\hat G$ contains $N$
as a normal subgroup, and an irreducible character
  $\tau \in \irr{\hat G}$ that extends $\theta$.  The next theorem explains exactly how
  to do this.

\begin{thm} Let $(G,N,\theta)$ be a character triple and let $\mathcal{P}$ be a projective representation of $G$ associated with $\theta$ such that the factor set $\alpha$ of $\mathcal{P}$ 
only takes roots of unity values. Let $Z\leq \mathbb{C}^\times$ be the subgroup generated by the values of $\alpha$. Let $\hat{G}=\{(g,z)\>|\> g\in G, z\in Z\}$ with the multiplication given as follows:

 $$(x,a)(y,b)=(xy,\alpha(x,y)ab).$$ Then $\hat{G}$ is a finite group. Besides, if we identify $N$ with $N\times 1$ and $Z$ with $1\times Z$, we have that the following hold.
\begin{enumerate}[label=(\alph*)]
\item $N\normal\hat{G}$, $Z\subseteq {\rm \textbf{Z}}(\hat{G})$, and $\hat{N}=N\times Z\normal \hat{G}$. Moreover, if $\pi:\hat G \rightarrow G$ is given by $(g,z) \mapsto g$, then $\pi$ is
an onto homomorphism with kernel $Z$. Also, if $N \subseteq \zent G$, then
$\hat N \subseteq \zent{\hat G}$.
\item The function $\hat{\mathcal{P}}(g,z)=z\mathcal{P}(g)$ 
defines an irreducible representation of $\hat{G}$ whose character $\tau\in{\rm Irr}(\hat{G})$ extends $\theta$.
In fact, $\tau(n,z)=z\theta(n)$ for $n \in N$ and $z \in Z$.
In particular,
 if $\hat{\theta}=\theta\times 1_Z\in{\rm Irr}(\hat{N})$, and $\hat{\lambda}\in{\rm Irr}(\hat{N})$ is defined by $\hat{\lambda}(n,z)=z^{-1}$, then $\hat{\lambda}$ is a linear $\hat{G}$-invariant character with $N={\rm ker}(\hat{\lambda})$ and $\hat{\lambda}^{-1}\hat{\theta}$ extends to $\tau\in{\rm Irr}(\hat{G})$.
\end{enumerate}
\label{teogtilde}
\end{thm}

\begin{proof}
See   Theorem 11.28 of \cite{Is} or Theorem 5.6 of \cite{nbook2}.
The properties of the factor set $\alpha$ that we have listed in Lemma \ref{lemafactorsets} are essential
to prove  (a).
\end{proof}

Given a character triple $(G,N,\theta)$, we call the group $\hat{G}$ defined in Theorem \ref{teogtilde} a \textbf{representation group for $(G,N,\theta)$ associated with $\mathcal{P}$}.
We also say that   $\tau \in \irr{\hat G}$ is the character of $\hat G$ {\bf associated} with $\mathcal P$.
By Theorem \ref{teogtilde}(b), we have that $\tau_N=\theta$.

\medskip

In order to define the $\theta$-blocks the following is essential. We assume the reader is familiar with the notion of character triple isomorphism (see Definition 11.23 of \cite{Is}). 
\begin{thm} Let $(G,N,\theta)$ be a character triple and let $\mathcal{P}$ be a projective representation of $G$ associated with $\theta$. Let $\hat{G}$ be a representation group for $(G,N,\theta)$ associated with $\mathcal{P}$. Then $(G,N,\theta)$ and $(\hat{G}/N,\hat{N}/N,\hat{\lambda})$ are  isomorphic character triples.
\label{lemaisomchartriplgtilde}
\end{thm}

\begin{proof}
See Theorem 11.28 of \cite{Is} or Corollary 5.9 of \cite{nbook2}.
\end{proof}

We shall frequently use how this character triple isomorphism
is constructed.  Let $\chi \in \irr{G|\theta}$.
We show how to construct $\chi^* \in \irr{\hat G/N|\hat\lambda}$.
Let $\pi: \hat G \rightarrow G$ be the homomorphism $(g,z) \mapsto g$, which has kernel $Z$.
Since $\pi$ induces an isomorphism $\hat G/Z \rightarrow G$, there is
a unique $\chi^{\pi} \in \irr{\hat G}$ such that $\chi^{\pi}(g,z)=\chi(g)$ for all $g \in G, z \in Z$.
Since $\chi$ lies over $\theta$ notice that $\chi^{\pi}$ lies over $\hat\theta=\theta \times 1_Z$,
and in particular over $\theta$.  Now by Theorem \ref{teogtilde}(b), the
character $\tau$ extends $\theta$.
By Gallagher's Corollary 6.17 of \cite{Is}, there exists a unique $\chi^* \in \irr{\hat G/N}$
such that $\chi^{\pi}=\chi^* \tau$. (Recall that we view the
characters of $H/N$ as characters of $H$ that contain $N$ in its kernel.)
Now, evaluating in $(1,z)$ for $z \in Z$, we easily check that $\chi^* \in \irr{\hat G/N|\hat\lambda}$.
The fact that $\chi \mapsto \chi^*$ defines an isomorphism of character triples
is the content of the proof of Theorem \ref{lemaisomchartriplgtilde}. (The
same construction can be done for every subgroup $N \le U \le G$ in place of $G$.)
Also, recall that $\hat\lambda(n,z)=z^{-1}$ for $n \in N$ and $z \in Z$.
\medskip

We say that the character triple $(\hat{G}/N,\hat{N}/N,\hat{\lambda})$ is a \textbf{standard isomorphic character triple} for $(G,N,\theta)$ given by $\mathcal P$, and that the bijective map
$^*: \irr{G|\theta} \rightarrow \irr{\hat G/N|\hat\lambda}$ that we have constructed, is the
$\textbf{standard bijection}$.

\medskip

\section{$\theta$-blocks}\label{secthetablocks}

If $(G,N,\theta)$ is a character triple, we are now ready to define the $\theta$-blocks
of $G$, and their $\theta$-defect groups.

\begin{defi}Let $(G,N,\theta)$ be a character triple.
Let $\hat G$ be a representation group
for $(G,N,\theta)$ and let  $\pi:\hat G \rightarrow G$ be the
canonical homomorphism $(g,z) \mapsto g$ with kernel $Z$. 
Let  $^*: \irr{G|\theta} \rightarrow
\irr{\hat G/N|\hat\lambda}$ be the associated standard
bijection. We say that a non-empty subset $B_\theta \subseteq{\rm Irr}(G|\theta)$
is a \textbf{$\theta$-block} of $G$
if there exists   a $p$-block $\hat{B}$ of $\hat{G}/N$
such that $$B_\theta^*=\{\chi^*\>|\> \chi\in B_\theta\}={\rm Irr}(\hat{B}|\hat\lambda)\, .$$
 If $\hat{D}/N$
is a defect group of $\hat{B}$, then
we say that $\pi(\hat D)/N$
is a {\bf $\theta$-defect group of $B_\theta$}. 

\end{defi}
 
Of course, 
note that the definition of $\theta$-blocks   depends on the choice of the standard isomorphic character triple and therefore on the choice of the projective representation associated with $\theta$. The same
happens with the $\theta$-defect groups.
Our main result in this section
is that $\theta$-blocks are in fact canonically defined, and that 
all the $\theta$-defect groups
are $G/N$-conjugate.
The following  result is the key to proving that.





\medskip

\begin{thm} Let $(G,N,\theta)$ be a character triple. Let $\mathcal{P}_1,\mathcal{P}_2$ be projective representations of $G$ associated with $\theta$, with factor sets $\alpha_1$ and $\alpha_2$, respectively,
whose values are roots of unity. 
Let $\hat{G}_i$ be the representation group associated with $\mathcal{P}_i$.
Let $(\hat{G}_1/N,\hat{N}_1/N,\hat\lambda_1)$ and $(\hat{G}_2/N,\hat{N}_2/N,\hat\lambda_2)$ be 
the standard isomorphic character triples given by $\mathcal{P}_1$ and $\mathcal{P}_2$,  respectively.
Let $\hat{G}=G\times Z_1\times Z_2$ and define the product
$$(g,z_1,z_2)(h,z_1',z_2')=(gh,\alpha_1(g,h)z_1z_1',\alpha_2(g,h)z_2z_2').$$

Then the following hold.

\begin{enumerate}[label=(\alph*)]
\item $\hat{G}$ is a finite group, $N\times 1\times 1$ is a normal subgroup of $\hat{G}$
(which we identify with $N$),
and $1 \times Z_1 \times Z_2$ is a central subgroup of $\hat G$
(which we identify with $Z_1 \times Z_2$). Also, $\hat{N}=N\times Z_1\times Z_2$ is a normal subgroup of $\hat{G}$ and $\hat{N}/N$ is central in $\hat{G}/N$.

\item The maps $\rho_1:\hat{G} \rightarrow \hat{G}_1$ and $\rho_2: \hat{G} \rightarrow \hat{G}_2$
given by $(g,z_1,z_2) \mapsto (g,z_1)$ and $(g,z_1,z_2) \mapsto (g,z_2)$ are surjective homomorphisms
with kernels $Z_2$ and $Z_1$, respectively.

\item
Suppose that $\tau_i \in \irr{\hat G_i}$ is the character associated with ${\mathcal P}_i$,
and let $\tau_i^{\rho_i} \in \irr{\hat G}$ be the corresponding character of $\hat G$.
Then there exists a linear character $\beta \in \irr{\hat G/N}$ such that
$$\tau_1^{\rho_1}=\beta \tau_2^{\rho_2} \, .$$

\item Let $\chi\in{\rm Irr}(G|\theta)$ and let $\chi_i^*\in \irr{\hat{G}_i/N|\hat{\lambda}_i}$ be
the image of $\chi$ under the standard bijection.
Let $\hat{\chi}_i=(\chi_i^*)^{\rho_i}\in{\rm Irr}(\hat{G}/N)$.
 Then  $\beta\hat{\chi}_1=\hat{\chi}_2$.  As a consequence, if $\hat{B}_i$ is the block 
 of $\hat{G}/N$ containing $\hat{\chi}_i$, then $\hat{B}_2=\beta\hat{B}_1$.
\item Let $B_i^*$ be the block of $\hat{G}_i/N$ containing $\chi_i^*$. Then the map 
$\psi\mapsto  \psi^{\rho_i}$ is a bijection from      
${\rm Irr}(B_i^*|\hat\lambda_i)$ to ${\rm Irr}(\hat{B}_i|\tilde{\lambda}_i)$, where  
$\tilde{\lambda}_1(n,z_1,z_2)=\hat\lambda_1(1,z_1)=z_1^{-1}$ and $\tilde{\lambda}_2(n,z_1,z_2)=\hat\lambda_2(1,z_2)=z_2^{-1}$
are linear characters of $\hat N/N$. 
 \item The map $ \psi \mapsto\beta \psi $ is a bijection from ${\rm Irr}(\hat{B}_1|\tilde{\lambda}_1)$ to ${\rm Irr}(\hat{B}_2|\tilde{\lambda}_2)$. In particular, $|{\rm Irr}(B_1^*|\hat\lambda_1)|= |{\rm Irr}(B_2^*|\hat\lambda_2)|$.
 
 \item Let  $\pi_i:\hat G_i \rightarrow G$ be the
canonical homomorphism $(g,z_i) \mapsto g$ with kernel $Z_i$.
 If $\hat D_i/N$ is defect group of $B_i^*$, then $\pi_1(\hat D_1)$ and $\pi_2(\hat D_2)$
 are $G$-conjugate.
\end{enumerate}
\label{lemaghat}
\end{thm}

\begin{proof} Using Lemma \ref{lemafactorsets}, parts (a) and (b) are straightforward.
We prove (c). Since $\mathcal{P}_1$ and $\mathcal{P}_2$ are projective representations of $G$ associated to $\theta$, by Theorem 11.2 of \cite{Is} we know that there exists $\xi:G\rightarrow\mathbb{C}^\times$ with $\xi(1)=1$,
constant on the cosets of $N$, such that $\mathcal{P}_2=\xi\mathcal{P}_1,$ and the factor sets $\alpha_1$ and $\alpha_2$ are related in this way

$$\alpha_2(g,h)=\alpha_1(g,h)\xi(g)\xi(h)\xi(gh)^{-1}.$$



Now $\tau_i\in{\rm Irr}(\hat{G}_i)$ is the character afforded by the irreducible representation $\hat{\mathcal{P}_i}$,
which is defined by $\hat{\mathcal{P}_i}(g,z_i)=z_i\mathcal{P}_i(g),$ for $z_i\in Z_i$ and $g\in G$. Then, using that $\mathcal{P}_2=\xi\mathcal{P}_1,$ we have that

\begin{align*}
\tau_1(g,z_1)=z_1z_2^{-1}\xi(g)^{-1}\tau_2(g,z_2)
\label{star}
\end{align*}
for $g \in G$ and $z_i \in Z_i$.
It is straightforward to prove that the function $\beta:\hat{G}\rightarrow\mathbb{C}^\times$ defined by $$\beta(g,z_1,z_2)=z_1z_2^{-1}\xi(g)^{-1}$$ is a linear character of $\hat G$
that contains $N$ in its kernel.

By definition, we have that   
  $\tau_1^{\rho_1}(g,z_1,z_2)=\tau_1(g,z_1)$
and $\tau_2^{\rho_2}(g,z_1,z_2)=\tau_2(g,z_2)$.
Therefore $\tau_1^{\rho_1}=\beta\tau_2^{\rho_2}$, as desired. This proves (c).

Let us denote by $\pi_i:\hat{G}_i \rightarrow G$ the
homomorphism $(g,z_i) \mapsto g$. Recall that, by definition,   $\chi_i^* \in \irr{\hat{G}_i/N}$ is the unique
character satisfying $\chi^{\pi_i}=\chi_i^* \tau_i$.
That is $$\chi(g)=\chi^{\pi_i}(g,z_i)=\chi_i^*(g,z_i)\tau_i(g,z_i)$$ for $g \in G$ and $z_i \in Z_i$.
By definition,
$\hat\chi_1(g,z_1,z_2)=\chi_1^*(g,z_1)$ and $\hat\chi_2(g,z_1,z_2)=\chi_2^*(g,z_2)$.
In particular, $\hat\chi_i \in \irr{\hat G}$ contains $N$ in its kernel.
Notice that we have that 
$\tau_1^{\rho_1}\hat{\chi}_1=\tau_2^{\rho_2}\hat{\chi}_2$.
Hence, 
$$\beta\hat\chi_1 \tau_2^{\rho_2}=\hat\chi_2 \tau_2^{\rho_2} \, .$$ 
Since $\tau_2^{\rho_2}$ extends $\theta\in{\rm Irr}(N)$
and $\beta\hat\chi_1,\hat\chi_2 \in \irr{\hat G/N}$, by Gallagher's Corollary 6.18 of \cite{Is}, we have that 
$$\beta\hat{\chi}_1=\hat{\chi}_2 \, .$$
Using Lemma \ref{lemamub1},     part (d) easily follows.





Next we prove part (e). Since $\rho_1(N)=N$,
then $\rho_1$ uniquely defines an onto homomorphism
$\tilde{\rho_1}:\hat G/N \rightarrow \hat G_1/N$ with
kernel $NZ_2/N \subseteq \zent{\hat G/N}$. Since $N \sbs \ker{\chi_1^*}$,
then notice that $\hat\chi_1=(\chi_1^*)^{\tilde\rho_1}$.  
Now $NZ_1/N$ is a subgroup of $\hat G_1/N$, and its inverse image under
$\tilde\rho_1$ is $\hat{N}/N$. Also, the character corresponding to $\hat\lambda_1$ under
$\tilde\rho_1$ is $\tilde\lambda_1$. By Corollary \ref{sameblock} (a) and (b),
we have that $\psi\mapsto  \psi^{\rho_i}$ is a bijection from      
${\rm Irr}(B_i^*|\hat\lambda_i)$ to ${\rm Irr}(\hat{B}_i|\tilde{\lambda}_i)$. 
(Notice that $ \psi^{\rho_i}= \psi^{\tilde\rho_i}$ because all of our characters have $N$ in their kernel).

Now we prove part (f). By using their definitions (and the fact  that $\xi(n)=1$ for $n \in N$),
we check that  $\beta_{\hat{N}}\tilde\lambda_1= \tilde\lambda_2$ . Therefore,
multiplication by the linear character $\beta$ sends bijectively
$\irr{\hat B_1|\tilde\lambda_1} \rightarrow \irr{\hat B_2|\tilde\lambda_2}$.

Finally, we prove part (g). As in part (e),
we have that   $\tilde{\rho_i}:\hat G/N \rightarrow \hat G_i/N$ 
is an onto homomorphism, with
central kernel, such that the map $\psi\mapsto  \psi^{\tilde\rho_i}$ is a bijection from      
${\rm Irr}(B_i^*|\lambda_i)$ to ${\rm Irr}(\hat{B}_i|\tilde{\lambda}_i)$.
Let $E_i/N$ be a defect group of $\hat{B}_i$. 
Since $\hat B_2=\beta\hat B_1$, we may assume that $E_i=E$ for $i=1,2$.
by Lemma \ref{lemamub1}. By Corollary \ref{sameblock}(c),
we have that $\tilde\rho_i(E/N)$ is a defect group of $B_i^*$.
Hence $\tilde\rho_i(E/N)=(\hat D_i/N)^{(g_i,1)}$ for some $g_i \in G$
(using that $Z_i$ is central in $\hat G_i$).
Now, since $\pi_i(N)=N$, we have that $\pi_i$ uniquely determines
an onto homomorphism $\tilde\pi_i: \hat G_i/N \rightarrow G/N$.
We easily check that
$\tilde \pi_1 \circ \tilde\rho_1= \tilde \pi_2 \circ \tilde\rho_2$.
Then 
$$\pi_1(\hat D_1)^{g_1}=\pi_2(\hat D_2)^{g_2} \, ,$$
as desired.
 \end{proof}

We can now prove the main result of this section. 

\begin{thm} Suppose that $N \nor G$, and $\theta \in \irr N$ is
$G$-invariant. Then the $\theta$-blocks of $G$ are well defined.
Furthermore, the
set of $\theta$-defect groups is a $G/N$-conjugacy class
of $p$-subgroups of $G/N$.
\label{thetablockswelldef}
\end{thm}

\begin{proof}

Let $(G,N,\theta)$ be a character triple and let $\mathcal{P}_1$ and $\mathcal{P}_2$ be projective representations associated with $\theta$. Let $(\hat{G}_1/N,\hat{N}_1/N,\hat\lambda_1)$ and $(\hat{G}_2/N,\hat{N}_2/N,\hat\lambda_2)$ be 
the standard isomorphic character triples given by $\mathcal{P}_1$ and $\mathcal{P}_2$ respectively.
Let $\pi_i:\hat G_i \rightarrow G$ be the homomorphism $\pi_i(g,z_i)=g$,
and let $\tau_i \in \irr{\hat G_i}$ be the character associated with $\mathcal{P}_i$.
Recall that if $\chi \in \irr{G|\theta}$, then $\chi^{\pi_i}=\chi_i^*\tau_i$,
for some uniquely defined $\chi_i^* \in \irr{\hat G_i/N}$.
The map $\chi \mapsto \chi_i^*$ from $\irr{G|\theta} \rightarrow \irr{\hat G_i/N|\hat\lambda_i}$
is the standard bijection.
 
   Let $A_1,A_2\subseteq{\rm Irr}(G|\theta)$ be such that $A_1^*=
   \{\varphi_1^*  \mid   \varphi \in A_1\}={\rm Irr}(B_1^*|\hat\lambda_1)$ and $A_2^*=
   \{\varphi_2^*  \mid  \varphi \in A_2\}={\rm Irr}(B_2^*|\hat\lambda_2),$ where $B^*_i$ is a block
  of $\hat G_i/N$. Suppose that $ \chi \in A_1 \cap A_2$.
  We wish to 
  to prove that $A_1=A_2$. 
  
  In order to do so, we construct the group $\hat G$ in Theorem \ref{lemaghat},
 and consider the group homomorphisms $\rho_i:\hat G \rightarrow \hat G_i$,
 in Theorem  \ref{lemaghat}(b). 
 By part (c) of this theorem, there is a linear character $\beta \in \irr{\hat G/N}$
 satisfying
  $$\tau_1^{\rho_1}=\beta \tau_2^{\rho_2} \, .$$
  As in Theorem \ref{lemaghat}(d), let
  $\hat{\chi}_i=(\chi_i^*)^{\rho_i}\in{\rm Irr}(\hat{G}/N)$,
  and let $\hat B_i$ be the block of $\hat G/N$ containing $\hat{\chi_i}$.
  By Theorem \ref{lemaghat}(d), 
  we have that $\hat B_2=\beta\hat B_1$.
  By Theorem \ref{lemaghat}(f), $|A_1^*|=|A_2^*|$, and therefore
  $|A_1|=|A_2|$.  We only need to prove that $A_1\subseteq A_2$, for instance.

 Let $\varphi_1\in A_1$. 
 Now, $\varphi_1^* \in A_1^*={\rm Irr}(B_1^*|\hat\lambda_1)$,
 and by Theorem \ref{lemaghat}(e) we have that $\hat\varphi_1=(\varphi_1^*)^{\rho_1}\in{\rm Irr}(\hat{B_1}|\tilde\lambda_1)$. By Theorem \ref{lemaghat}(f), $\beta\hat{\varphi_1}\in{\rm Irr}(\hat{B}_2|\tilde{\lambda}_2)$. By Theorem \ref{lemaghat}(e), let $\varphi_2\in A_2$ be such that $\beta\hat{\varphi_1}=(\varphi_2^*)^{\rho_2}$. We claim that $\varphi_1=\varphi_2$.   Recall that $\tau_i\varphi_i^*=\varphi_i^{\pi_i}$and that
 $\tau_1^{\rho_1}=\beta\tau_2^{\rho_2}$.  If $g \in G$, then we have that

\begin{align*}
\varphi_1(g)&=\varphi^{\pi_1}_1(g,1)=\tau_1(g,1)\varphi_1^*(g,1)\\
&=\tau_1^{\rho_1}(g,1,1)\hat{\varphi}_1(g,1,1)\\
&=\beta(g,1,1)\tau_2^{\rho_2}(g,1,1)\hat{\varphi}_1(g,1,1)\\
&=\tau_2(g,1)(\beta\hat{\varphi}_1)(g,1,1)\\
&=\tau_2(g,1)(\varphi_2^*)^{\rho_2}(g,1,1)\\
&=\tau_2(g,1)\varphi_2^*(g,1)\\
&={\varphi}_2^{\pi_2}(g,1)\\
&=\varphi_2(g),
\end{align*}
as desired. This completes the proof of the first part of the theorem.
The second part easily follows from Theorem \ref{lemaghat}(g).
It is elementary to show that the $\theta$-defect groups are $p$-subgroups of $G/N$.
\end{proof}


\section{Properties of $\theta$-blocks}\label{secproperties}
The following gives us some key properties of $\theta$-blocks.
 
\begin{thm}\label{properties} Let $(G,N,\theta)$ be a character triple.
Let $B_\theta$ be a $\theta$-block of $G$, and let $D_\theta/N$
be a $\theta$-defect group of $B_\theta$.
\begin{enumerate}[label=(\alph*)]
 
 \item
 There is
a $p$-block $B$ of $G$ such that
$B_\theta$ is contained in the set ${\rm Irr}(B|\theta)$.
Also, there is a defect group $D$ of  $B$ such that $D_\theta \subseteq DN$.

\item
If $N \sbs \zent G$, then there is a $p$-block
$B$ of $G$ and a defect group $D$ of $B$ such that   $B_\theta=\irr{B|\theta}$,
and $D_\theta=DN$.

\item
If $\theta$ has an extension $\chi \in \irr G$, then there
is a $p$-block $\bar B$ of $G/N$
and a defect group $DN/N$ of $\bar B$
 such that $B_\theta=\{ \gamma \chi \mid \gamma \in \irr{\bar B}\}$
 and $D_\theta/N=DN/N$.
 
 \item
 If $G/N$ is a $p$-group, then $B_\theta=\irr{G|\theta}$ and $D_\theta/N=G/N$.
\end{enumerate}
\end{thm}
\begin{proof}

Let $\hat G$ be a representation group associated with $(G,N,\theta)$, 
with associated character $\tau \in \irr{\hat G}$.  Recall that $\tau_N=\theta$.
Let $^*:\irr{G|\theta} \rightarrow \irr{\hat G/N|\hat\lambda}$
be the standard bijection. Let $\pi:\hat G \rightarrow G$
be the homomorphism $(g,z) \mapsto g$.
Since $\pi(N)=N$, let  $\hat \pi:\hat G/N \rightarrow G/N$ be
 the corresponding onto homomorphism. Notice that $\hat G/N$
 is a central extension of $G/N$.

By definition, there is a
Brauer $p$-block
$\hat B$ of $\hat G/N$ such that
$(B_\theta)^*=\irr{\hat B|\hat\lambda}$.
 Recall that $\chi^\pi=\chi^* \tau$ for $\chi \in \irr{G|\theta}$.

Now, fix $\chi\in B_\theta$ and let $B$ be the $p$-block of $G$ containing $\chi$. We claim that $B_\theta\subseteq{\rm Irr}(B|\theta)$. Indeed, let  $\psi\in B_\theta$.
Then $\chi^*,\psi^*\in \hat B$. 
Since $\tau_N=\theta$ and
 ${\chi}^{\pi}=\tau\chi^*$, ${\psi}^{\pi}=\tau\psi^*$, by Lemma \ref{lema}
we have that $\chi^\pi$ and $\psi^\pi$ lie in the same 
$p$-block of $\hat G$. 
By Corollary \ref{sameblock}, $\chi, \psi$ lie in the same $p$-block of $G$.
This proves the first part of  (a).
If $\hat D/N$ is a defect group of $\hat B$, by Lemma \ref{lema}
we have that $\hat D \subseteq EN$ for some defect group
$E$ of the block of $\chi^\pi$. Now, $\pi(E)$ is a defect group of the block of $\chi$
by Corollary \ref{sameblock}(c),
and $\pi(\hat D) \subseteq \pi(E)N$. This proves the second part of (a).
Notice now that if $N$ is central,
then $\tau$ is linear and the defect groups of the block of $\chi^\pi=\tau\chi^*$ are the
the defect groups of the block of $\chi^*$ (multiplying by $\tau^{-1}$ and using Lemma
\ref{lemamub1}).  Since $N$ is central in $\hat G$ by Theorem \ref{teogtilde}(a),
we have that $\hat D=EN$ by Lemma \ref{lemaliftblock}(b).

Next, we complete the
proof of part (b). Suppose $N$ is central  and that $\gamma \in \irr{B|\theta}$.
In particular, $\tau$ is linear.
Write $\gamma^\pi=\gamma^* \tau$, for some $\gamma^* \in \irr{\hat G/N|\hat\lambda}$.
Now, since $\gamma$ and $\chi$ lie in the same $p$-block of $G$,
we have that $\gamma^\pi$ and $\chi^\pi$ lie in the same $p$-block of $\hat G$
by Corollary \ref{sameblock}. Therefore $\gamma^*\tau$ and $\chi^*\tau$
lie in the same $p$-block of $\hat G$. By Lemma \ref{lemamub1}, 
multiplying by $\tau^{-1}$, we have that $\gamma^*$ and $\chi^*$ lie
in the same $p$-block of $\hat G$.
Now, $N \sbs \zent{\hat G}$, by Theorem \ref{teogtilde}(a).
Thus $\gamma^*$ and $\chi^*$ lie
in the same $p$-block of $\hat G/N$ by Lemma \ref{lemaliftblock} (a).
Hence $\gamma^* \in \irr{\hat B|\hat\lambda}$, and therefore $\gamma \in B_\theta$.
This proves (b). (The part on the defect groups 
follows from the previous paragraph.)

For part (c), notice that 
if $\mathcal P$ is a representation affording $\chi$, then 
$\mathcal P$ is a projective representation associated with $(G,N,\theta)$ with trivial
factor set. Hence $\hat G=G$ is a representation group for $(G,N,\theta)$ with
associated character $\tau=\chi$. In this case the standard bijection is the map $\beta\chi\mapsto \beta$ from ${\rm Irr}(G|\theta)\rightarrow {\rm Irr}(G/N)$ given by Gallagher's Corollary 6.18 of \cite{Is}, and part (c) easily follows.

Next we prove part (d). Let $(\hat{G}/N,\hat{N}/N,\hat{\lambda})$ be a standard isomorphic character triple and let $^*:{\rm Irr}(G|\theta)\rightarrow{\rm Irr}(\hat{G}/N|\hat{\lambda})$ be the standard bijection. Let $\hat{B}$ be the $p$-block of $\hat{G}/N$ such that $(B_\theta)^*={\rm Irr}(\hat{B}|\hat{\lambda})$. By Theorem 9.2 and Corollary 9.6 of \cite{nbook} we have that ${\rm Irr}(\hat{B}|\hat{\lambda})={\rm Irr}(\hat{G}/N|\hat{\lambda})$. Therefore $|B_\theta|=|(B_\theta)^*|=|{\rm Irr}(\hat{B}|\hat{\lambda})|=|{\rm Irr}(\hat{G}/N|\hat{\lambda})|=|{\rm Irr}(G|\theta)|$, and the first part of (d) is proved. Let $D_\theta/N$ be a $\theta$-defect group of $B_\theta$ and let $\hat{D}/N\leq \hat{G}/N$ be a defect group of $\hat{B}$ such that $\pi(\hat{D})/N=D_\theta/N$.  Recall that $\hat{\pi}:\hat{G}/N\rightarrow G/N$ defined by $(g,z)N\mapsto gN$ is an onto homomorphism with ${\rm ker}(\hat{\pi})=\hat{N}/N$. Write $G^*=\hat{G}/N$, $N^*=\hat{N}/N$ and $D^*=\hat{D}/N$, and write $\tilde{\pi}:G^*/N^*\rightarrow G/N$ for the induced isomorphism.  Then $\tilde\pi(D^*N^*/N^*)=D_\theta/N$. Let $K^*\in{\rm cl}(G^*)$ be a defect class for $\hat{B}$. By Corollary 3.8 of \cite{nbook}, we know that $K^*$ consists of $p$-regular elements. Since $G^*/N^*$ is a $p$-group, we have that $K^*\subseteq N^*\subseteq{\rm \textbf{Z}}(G^*)$. Let $x^*\in K^*$ be such that $D^*\in{\rm Syl}_p({\rm \textbf{C}}_{G^*}(x^*))$. Since $K^*$ is central, we have that $D^*\in{\rm Syl}_p(G^*)$. Then $D^*N^*/N^*\in{\rm Syl}_p(G^*/N^*)$ and, since $\tilde{\pi}(D^*N^*/N^*)=D_\theta/N$ we have that $D_\theta/N\in{\rm Syl}_p(G/N)$. Since $G/N$ is $p$-group, $D_\theta/N=G/N$. 



\end{proof}

If $(G,N,\theta)$ is a character triple
and $B_\theta \sbs \irr{G|\theta}$ is
a $\theta$-block, then, in general, $B_\theta$ is
much smaller than the set $\irr{B|\theta}$,
where $B$ is the Brauer $p$-block
containing $B_\theta$.
Indeed, if $G$ is a $p$-constrained group, for instance, 
$N \nor G$ 
is such that $p$ does not divide $|G/N|$, and $\theta \in \irr N$ extends to $G$,
then we have that $G$ has only one $p$-block $B$. Thus $\irr{B|\theta}=
\irr{G|\theta}$, while the $\theta$-blocks
have size 1 (by Theorem \ref{properties}(c)).
\medskip

To end this section,
we give an analogue of a classical result on blocks.
  
\begin{thm} Let $\chi\in{\rm Irr}(G|\theta)$ and let $B_\theta$ be the $\theta$-block containing $\chi$. 
Let $g \in G$ and suppose that $(gN)_p$ is not $G/N$-conjugate to any element of $D_\theta/N$, where $D_\theta/N$ is a $\theta$-defect group of $B_\theta$. Then $\chi(g)=0$.
\end{thm}

\begin{proof}
Let $(\hat{G}/N,\hat{N}/N,\hat\lambda)$ be a standard isomorphic character triple of $(G,N,\theta)$. Write $\pi:\hat{G}\rightarrow G$ for the canonical onto homomorphism. Since $\pi(N)=N$, $\pi$ induces a homomorphism $\hat{\pi}:\hat{G}/N\rightarrow G/N$ with ${\rm ker}(\hat{\pi})=\hat{N}/N$. Write $G^*=\hat{G}/N$ and $N^*=\hat{N}/N$ and write $\tilde{\pi}:G^*/N^*\rightarrow G/N$ for the induced isomorphism.  

Let $gN\in G/N$ and let $g^*N^*\in G^*/N^*$ such that $\tilde{\pi}(g^*N^*)=gN$. Let $\hat{B}$ be the $p$-block of $G^*$ such that $(B_\theta)^*={\rm Irr}(\hat{B}|\hat{\lambda})$, where $^*:{\rm Irr}(G|\theta)\rightarrow{\rm Irr}(G^*|\hat{\lambda})$ is the standard bijection. Let $D^*=\hat{D}/N$ be the defect group of $\hat{B}$ such that $\pi(\hat{D})/N=D_\theta/N$. Notice that $\tilde{\pi}(D^*N^*/N^*)=D_\theta/N$. 

Since $(gN)_p$ is not $G/N$-conjugate to any element of $D_\theta/N$, we have that $(g^*)_pN^*$ is not $G^*/N^*$-conjugate to any element of $D^*N^*/N^*$. Hence $(g^*)_p$ is not contained in any defect group of the block of $\chi^*$. By Corollary 5.9 of \cite{nbook} we have that $\chi^*(g^*)=0$. Recall that $\chi^{\pi}=\tau\chi^*$, where $\tau\in{\rm Irr}(\hat{G})$ is the character associated to $\mathcal{P}$.  Since $\tilde{\pi}(((g,1)N)(\hat{N}/N))=\hat{\pi}((g,1)N)=gN$, we have that $g^*=(g,1)N$ and then $$\chi(g)=\chi^\pi(g,1)=\tau(g,1)\chi^*((g,1)N)=\tau(g,1)\chi^*(g^*)=0.$$



\end{proof}

\section{Theorem D}\label{secthmD}

In this section we prove Theorem D of the introduction. We will need the following result,
which is essentially a result of R. Kn\"orr.
Recall that if $(G,N, \theta)$ is a character triple, then $xN \in G/N$
is {\bf $\theta$-good} if $\theta$ has a $D$-invariant extension
to $N\langle x\rangle$, where $D/N=\cent{G/N}{xN}$.
The $\theta$-good conjugacy classes of $G/N$ (those consisting of $\theta$-good
elements) play the role of the conjugacy classes of $G$ when we are working
with characters of $G$ over $\theta$. For instance, it is a theorem of P. X. Gallagher
 that $|\irr{G|\theta}|$ is the number of
conjugacy classes of $G/N$ consisting of $\theta$-good elements
(see Theorem 5.16 of \cite{nbook2}).

\begin{thm}
Suppose that $Z\subseteq{\rm \textbf{Z}}(G)$ and let $\theta\in{\rm Irr}(Z)$. Suppose that $gZ$ and $hZ$ are not $G/Z$-conjugate. Then 

$$\sum_{\chi\in{\rm Irr}(G|\theta)}\chi(g)\chi(h^{-1})=0.$$ Also $$\sum_{\chi\in{\rm Irr}(G|\theta)}|\chi(g)|^2=|{\bf C}_{G/Z}(gZ)|$$ if $g$ is $\theta$-good.
\label{teoorthogonality}
\end{thm}

\begin{proof}
The first part is a special case of Corollary 7 of \cite{knorr}.
The second part is an unpublished result of 
I. M. Isaacs. For a proof see Theorem 5.21 of \cite{nbook2}.
\end{proof}

\medskip

As we said in Section \ref{secreviewpblocks}, we
write $G^{p'}$ for the set of $p$-regular elements of
the finite group $G$. If $\chi\in{\rm Irr}(G)$, we denote by $\chi^{p'}$ the restriction of $\chi$ to $G^{p'}$. We know that we can write

$$\chi^{p'}=\sum_{\varphi\in{\rm IBr}(G)}d_{\chi\varphi}\varphi$$ for uniquely determined non-negative integers $d_{\chi\varphi}$ called the decomposition numbers. The matrix $D=(d_{\chi\varphi})$ is called the decomposition matrix of $G$.
The following is Theorem D of the introduction.

\begin{thm}\label{thmE}
Suppose that $Z$ is a central subgroup of $G$, and let $\theta \in \irr Z$. 
Let $B$ be a Brauer $p$-block of $G$ such that $\irr{B|\theta}$
is not empty. Then the matrix
$D_{B,\theta}=(d_{\chi \varphi})$, where $\chi \in \irr{B|\theta}$ and $\varphi \in \ibr{B}$
is not of the form
$$\left(\begin{array}{cc}
*&0\\
0&*\\
\end{array}\right).$$
\end{thm}

\begin{proof}

 Let $D=(d_{\chi\varphi})$ be the decomposition matrix of $G$ and let $M_\theta$ be the submatrix of $D$ whose rows are indexed by the characters in ${\rm Irr}(G|\theta)=
 \{\chi_1, \ldots, \chi_k\}$.
 Let $\{x_1,x_2,\ldots,x_l\}$ be a set of representatives of the $p$-regular conjugacy classes of $G$.  Let $X_{\theta}=(\chi_i(x_j))$ be the submatrix of the character table of $G$ with rows indexed by elements in $\irr{G|\theta}$
   and columns indexed by
 the representatives of the
 $p$-regular conjugacy classes of $G$. Let $Y=(\varphi_i(x_j))$ be the Brauer character table of $G$, where $\ibr G=\{ \varphi_1, \ldots, \varphi_l\}$. Then we have that $X_{\theta}=M_\theta Y$.

We first assume that $Z$ is a $p$-group.
Suppose that $g\in G$ is $p$-regular. We claim that $g$ is $\theta$-good. First we prove that ${\rm \textbf{C}}_{G/Z}(gZ)={\rm \textbf{C}}_G(g)/Z$. Indeed, let $xZ\in {\rm \textbf{C}}_{G/Z}(gZ)$. Then, 

$$gZ=(gZ)^{xZ}=g^xZ,$$ and therefore $g^x=gz$ for some $z\in Z$. Since $z$ is a central $p$-element and $g^x$ and $g$ are $p$-regular elements, we have that $z=1$ and therefore $x\in{\rm \textbf{C}}_G(g)$. Now let $\eta$ be an extension of $\theta$ to $\langle Z,g\rangle$. We need to prove that $\eta$ is ${\rm \textbf{C}}_G(g)$-invariant. But this is clear since ${\rm \textbf{C}}_G(g)\subseteq{\rm \textbf{C}}_G(x)$ for all $x\in\langle Z,g\rangle$. Hence $g$ is $\theta$-good and the claim is proven.

Note that if $x_i$ and $x_j$ are not $G$-conjugate $p$-regular elements, then $x_iZ$ and $x_jZ$ are not $G/Z$-conjugate. Indeed, suppose that there exists $gZ\in G/Z$ such that $x_iZ=(x_jZ)^{gZ}=x_j^gZ$, hence $x_i=x_j^gz$ for some $z\in Z$. Again, since $Z$ is a central $p$-group and $x_i$ and $x_j^g$ are $p$-regular elements, we have that $z=1$ and hence $x_i=x_j^g$. 
 Let $E\in{\rm Mat}_{l}(\mathbb{C})$ be the diagonal matrix with diagonal entries $|{\rm \textbf{C}}_{G/Z}(x_iZ)|$. By Theorem \ref{teoorthogonality} we have that 


$$E=X_{\theta}^t\overline{X}_{\theta}=Y^t(M_\theta)^tM_\theta\overline{Y}.$$ 

What we have done until now holds for every $\theta \in \irr Z$.
If $\theta=1_Z$ is the trivial
character of $Z$,
notice that $M_{1_Z}$ is the decomposition matrix of $G/Z$, since ${\rm Irr}(G|1_Z)={\rm Irr}(G/Z)$ and ${\rm IBr}(G/Z)={\rm IBr}(G)$ by Theorem 7.6 of \cite{nbook}.
By the previous equation, for $\theta=1_Z$,
we have

$$E=X_{1_Z}^t\overline{X}_{1_Z}=Y^t(M_{1_Z})^tM_{1_Z}\overline{Y},$$ and since $Y$ is a regular matrix, we conclude that $$C=M_{1_Z}^tM_{1_Z}=M_{\theta}^t M_\theta,$$ where $C$ is the Cartan matrix of $G/Z$.
Until now, our ordering of ${\rm Irr}(G|\theta)=
 \{\chi_1, \ldots, \chi_k\}$ and $\ibr G=\{ \varphi_1, \ldots, \varphi_l\}$ was arbitrary.
 Now let $B_1,B_2,\ldots,B_r$ be the different  $p$-blocks of $G$,
 and order $\irr{G|\theta}$ and $\ibr G$ by blocks
 (so that the first characters are in $B_1$, and so on).
 Since $Z$ is a central $p$-group, by Theorem 7.6 of \cite{nbook} we have that there exists a unique $p$-block $\overline{B}_i$ of $G/Z$ contained in $B_i$. Let $C_{\overline{B}_i}$ be the Cartan matrix of $\overline{B}_i$. We have that $C={\rm diag}(C_{\overline{B}_1},\ldots,C_{\overline{B}_r})$ and $M_\theta={\rm diag}(M_{B_1,\theta},\ldots,M_{B_r,\theta})$ are block diagonal
 matrices. Then,

$$M_\theta^tM_\theta=\diag(M_{B_1,\theta}^tM_{B_1,\theta},\ldots,M_{B_r,\theta}^tM_{B_r,\theta}).$$ 

\medskip

Since $C=M_\theta^tM_\theta$, we 
necessarily have that $C_{\overline{B}_i}=M_{B_i,\theta}^tM_{B_i,\theta}$
for every $i$. Now if  $M_{B, \theta}$ is of the form 


$$\left(\begin{array}{cc}
*&0\\
0&*\\
\end{array}\right),$$ so is $C_{\overline{B}}$. By Problem 3.4 of \cite{nbook} this is a contradiction. 
 
\medskip

This ends the proof of the case where $Z$ is a $p$-group. We prove now the general case. Write $Z=Z_p\times Z_{p'}$, where $Z_p$ is the Sylow $p$-subgroup of $Z$ and write $\theta=\theta_p\times\theta_{p'}$, with $\theta_p\in{\rm Irr}(Z_p)$ and $\theta_{p'}\in{\rm Irr}(Z_{p'})$. By assumption, there is $\chi \in \irr B$ over $\theta$, and therefore over $\theta_{p'}$.
Now, $B$ covers the block $\{\theta_{p'}\}$, and by Theorem 9.2 of \cite{nbook},
we have that $\irr{B|\theta_{p'}}=\irr B$. 
We conclude that $D_{B,\theta}=D_{B,\theta_p}$, and 
we are done by the central $p$-group case.\end{proof}

Notice, for instance, that if $G={\sf A}_4$,
$p=2$, $Z \nor G$ is the Klein subgroup, and $\theta=1_Z$,
then $G$ has a unique 2-block, and the matrix
$D_{B,\theta}$ is the identity.

\section{Theorems B and C}\label{secthmC}

We start this section by proving Theorem C of the introduction.
Recall that Brauer's $k(B)$-conjecture asserts that
if $B$ is a block with defect group $D$, then  $k(B)=|\irr B|\le |D|$.
The key is to use Theorem C of \cite{nlocglo}.

\begin{thm} \label{thmC}The $k(B)$-conjecture is true 
for every finite group if and only if for every
character triple $(G,N,\theta)$, we have that every
$\theta$-block $B_\theta$ has size less than or equal the size of any of its $\theta$-defect groups.
\end{thm}
\begin{proof} Let $(G,N,\theta)$ be a character triple and let  $(\hat{G}/N,\hat{N}/N,\hat{\lambda})$ be a standard isomorphic character triple. Write $G^*=\hat{G}/N$, $N^*=\hat{N}/N$ and $\theta^*=\hat{\lambda}$. Let $B_\theta$ be a $\theta$-block and let $D_\theta/N$ be a $\theta$-defect group of $B_\theta$. Suppose first that the $k(B)$-conjecture holds for every finite group. Write $N^*=N^*_p\times N^*_{p'}$ where $N^*_p\in{\rm Syl}_p(N^*)$, and write $\theta^*=\theta^*_p\times \theta^*_{p'}$, with $\theta^*_p\in{\rm Irr}(N^*_p)$ and $\theta^*_{p'}\in{\rm Irr}(N^*_{p'})$. From the definition of the $\theta$-blocks, we have that there exists a $p$-block $B^*$ of $G^*$ such that 
$ B_\theta^*= {\rm Irr}(B^*|\theta^*)$, where $^*:\irr{G|\theta}
\rightarrow \irr{G^*|\theta^*}$ is the standard bijection.
Thus $$|B_\theta|=|{\rm Irr}(B^*|\theta^*)|\, .$$
Now by Theorem 9.2 of \cite{nbook}, and using that
${\rm Irr}(B^*|\theta^*) \ne \emptyset$ (by the definition
of $\theta$-blocks), we have that 
 ${\rm Irr}(B^*|\theta^*)=\irr{B^*|\theta^*_{p}}$.
 By Theorem C of \cite{nlocglo} we have that $|{\rm Irr}(B^*|\theta^*_p)|\leq |{\rm Irr}(\overline{B^*})|$, where $\overline{B^*}$ is the unique $p$-block of $G^*/N^*_p$ contained in $B^*$. Let $D^*$ be a defect group of $B^*$. By Theorem 9.10 of \cite{nbook} we have that $D^*/N^*_p$ is a defect group of $\overline{B^*}$.  Since the $k(B)$ conjecture holds for $G^*/N^*_p$ we have that $|{\rm Irr}(\overline{B^*})|\leq |D^*/N^*_p|=|D^*N^*/N^*|$. 
  Now, write $D^*=\hat D/N$, for some
  subgroup $\hat D$ of $\hat G$, and by definition,
  recall that $\pi(\hat D)/N=D_\theta/N$ is a $\theta$-defect group of $B_\theta$,
  where $\pi: \hat G \rightarrow G$ is the onto homomorphism $(g,z) \mapsto g$.
  It is then enough to show that 
  $|D^*N^*/N^*|=|D_\theta/N|$.
  Notice that $\hat{\pi}:G^*\rightarrow G/N$ defined by $(g,z)N\mapsto gN$ is an onto group homomorphism with kernel $N^*$. Write $\tilde{\pi}:G^*/N^*\rightarrow G/N$ for the isomorphism induced by $\hat{\pi}$,
  and notice that $\tilde\pi(D^*N^*/N^*)=\hat\pi(D^*)=\pi(\hat D)/N=D_\theta/N$. Then 
  $D_\theta/N$ and $D^*N^*/N^*$ are isomorphic,
  and  $|D_\theta/N|=|D^*N^*/N^*|$, as desired.

For the converse, simply take $N=1$ and apply Theorem \ref{properties}(c).
\end{proof}

Next we prove Theorem B of the introduction. The key is Theorem 3 of \cite{sambale}. Recall that Conjecture A asserts the following: if $B_\theta \sbs \irr{G|\theta}$ is
a $\theta$-block with $\theta$-defect group $D_\theta/N$ and $\theta$
extends to $D_\theta$, then $(\chi(1)/\theta(1))_p=|G:D_\theta|_p$ for all $\chi \in B_\theta$
if and only if $D_\theta/N$ is abelian.

\begin{thm}
Conjecture A and Brauer's Height Zero conjecture are equivalent.
\end{thm}

\begin{proof}
Let $(G,N,\theta)$ be a character triple, and let $B_\theta$ be a $\theta$-block
with $\theta$-defect group $D_\theta/N$. As in the proof of Theorem \ref{thmC},
let $(G^*,N^*,\theta^*)$ be a standard isomorphic triple, with standard bijection
$^*: \irr{G|\theta} \rightarrow \irr{G^*|\theta^*}$, and suppose
that $B^*$ is the block of $G^*$ such that $(B_\theta)^*=\irr{B^*|\theta^*}$.
Recall  that $N^* \sbs \zent{G^*}$. We have shown above that $D_\theta/N$
is isomorphic to $D^*N^*/N^*$, where $D^*$ is a defect group of $B^*$.
(In fact, we have shown that if $\tilde\pi:G^*/N^* \rightarrow G/N$
is the group isomorphism induced by $\pi:\hat G \rightarrow G$, then $\tilde\pi(D^*N^*/N^*)=D_\theta/N$.)
Notice that, since $N^*$ is central in $G^*$, we have that the Sylow
$p$-subgroup of $N^*$ is contained in $D^*$ (Theorem 4.8 of \cite{nbook}),
and therefore $|D^*N^*:D^*|_p=1$. Then
$$|G:D_\theta|_p=|G/N:D_\theta/N|_p=|G^*:D^*N^*|_p=|G^*:D^*|_p \, .$$
Now, as is well-known, character triple isomorphisms preserve ratios
of character degrees (see Lemma 11.24 of \cite{Is}), that is $\chi(1)/\theta(1)=\chi^*(1)/\theta^*(1)=\chi^*(1)$
for $\chi \in \irr{G|\theta}$. 
In particular, if $\chi \in B_\theta$,
then
$$(\chi(1)/\theta(1))_p=\chi^*(1)_p=|G^*:D^*|_pp^{h(\chi^*)}=|G:D_\theta|_p p^{h(\chi^*)}\, ,$$
where $0 \le h(\chi^*)$ is the height of $\chi^*$ in $B^*$.

By the properties of
character triple isomorphisms, notice that $\theta$ extends to $D_\theta$
if and only if $\theta^*$ extends to $D^*N^*$. Now, it is clear that Conjecture A and the projective version of the Height Zero conjecture due to G. Malle and G. Navarro (Conjecture A of \cite{mn}) are equivalent. By Theorem 3 of \cite{sambale}, we are done. 
\end{proof}

We end this paper with the following observation.

\begin{pro} Conjecture A implies the Gluck-Wolf-Navarro-Tiep theorem.
\end{pro}
\begin{proof}

Suppose that $p$ does not divide $\chi(1)/\theta(1)$ for all $\chi\in{\rm Irr}(G|\theta)$. Let $B_\theta$ be a $\theta$-block and let $\chi\in B_\theta$. Since $(\chi(1)/\theta(1))_p=|G:D_\theta|_p p^{h(\chi^*)}$, we have that $|G:D_\theta|$
is not divisible by $p$. Hence $D_\theta/N$ is a Sylow
$p$-subgroup of $G/N$. Since $p$ does not divide $\chi(1)/\theta(1)$,
and all the irreducible constituents of $\chi_{D_\theta}$ lie
over $\theta$, it follows that there is some irreducible
constituent $\gamma \in \irr{D_\theta|\theta}$ such that $p$ does
not divide $\gamma(1)/\theta(1)$. By Corollary 11.29 of \cite{Is},
we have that $\gamma_N=\theta$. Since we are assuming that Conjecture A holds, we have that $D_\theta/N$ is abelian. Since $D_\theta/N\in{\rm Syl}_p(G/N)$, we have that $G/N$ has abelian Sylow $p$-subgroups.

\end{proof}

\end{document}